\documentclass[10pt, oneside,reqno]{amsart}
\usepackage{amsmath, amssymb, amsfonts}
\usepackage{amsthm}
\usepackage{hyperref}
\theoremstyle{definition}
\newtheorem{thm}[equation]{Theorem}

\newtheorem{prop}[equation]{Proposition}

\numberwithin{equation}{section}
\begin{document}
\author{Kamran Alam Khan}
\address{Department of Mathematics \\
 V. R. A. L. Govt. Girls P. G. College \\
 Bareilly (U.P.)-INDIA}
\email{kamran12341@yahoo.com}
\title{Generalized normed spaces and Fixed point Theorems}
\pagestyle{myheadings}
\thispagestyle{empty}
\begin{abstract}
\noindent
 G\"{a}hler (\cite{GA1}, \cite{GA2}) introduced and investigated the notion of \textit{2-metric spaces} and \textit{2-normed spaces} in sixties. These concepts are inspired by the notion of area in two dimensional Euclidean space. In this paper, we choose a fundamentally different approach and introduce a possible generalization of usual norm retaining the distance analogue properties. This generalized norm will be called as \emph{$G$-norm}. We show that every $G$-normed space is a $G$-metric space and therefore, a topological space and develop the theory for $G$-normed spaces. We also introduce $G$-Banach spaces and obtain some fixed point theorems.
\end{abstract}

\maketitle

\section{INTRODUCTION}
G\"{a}hler (\cite{GA1}, \cite{GA2}) introduced and studied the concept of  \textit{2-metric spaces} and \textit{2-normed spaces} and extended the theory to \textit{$n$-normed spaces} in (\cite{GA3}, \cite{GA4}, \cite{GA5}). Since then many authors(\cite{FR1}, \cite{FR2}, \cite{GUN}, \cite{HA1}, \cite{MIS}, \cite{RIY} etc.)  have published a number of articles devoted to these concepts. It was mentioned by G\"{a}hler \cite{GA1} that the notion of a 2-metric is an extension of an idea of ordinary metric. The usual metric is a kind of generalization of the notion of distance whereas the concept of a 2-metric and hence that of a 2-norm are inspired by the notion of area in two dimensional Euclidean space and geometrically $d(x,y,z)$ represents the area of a triangle formed by the points $x$, $y$ and $z$ in $X$ as its vertices. But this is not always true. Sharma \cite{SHA} showed that $d(x,y,z)=0$ for any three distinct points $x$, $y$, $z$ $\in \mathbb{R}^2$. Also  K. S. Ha et al \cite{HA1} have shown that in many cases there is no connection between the results obtained in the usual metric spaces and 2-metric spaces.\\  
\newline \indent  B. C. Dhage \cite{DHA} attempted to generalize the concept of usual metric and introduced the concept of $D$-metric in order to translate results from usual metric space to $D$-metric space. But the topological structure of $D$-metric spaces was found to be incorrect (\cite{MU1}). Finally, Mustafa and Sims (\cite{MU2}) introduced the concept of \textit{$G$-metric} in which the tetrahedral inequality is replaced by an inequality involving repetition of indices. This new approach is fundamentally different from that of G\"{a}hler and retains the notion of distance. Recently the author (\cite{KHA}) generalized the concept to $n (\ge 3) $ variables and introduced \textit{Generalized n-metric spaces}. In this paper our aim is to generalize the concept of normed space in such a manner that the generalized norm retains the distance analogue properties of the usual norm. We call this generalized norm a \textit{$G$-norm}. We show that every $G$-normed space is a $G$-metric space and therefore, a topological space. Hence the topological concepts such as open subset, closed subset, limit, closure etc make sense. We develop the theory for $G$-normed spaces and also introduce $G$-Banach spaces. Finally we obtain some fixed point theorems.\\
Let us begin with some definitions-

\subsection{Definition} (\cite{GA2})
Let $X$ be a real linear space of dimension greater than 1 and let $\|.,.\|$ be a real valued function on $X\times X$ satisfying the following conditions:
\begin{enumerate}
\item $\|x,y\|\ge 0$ for every $x,y\in X$; $\|x,y\|=0$ if and only if $x$ and $y$ are linearly dependent,
\item $\|x,y\|=\|y,x\|$ for every $x,y\in X$,
\item $\| \alpha x,y\|=|\alpha|\, \|x,y\|$ for every $x,y\in X$ and $\alpha \in \mathbb{R}$,
\item $\|x,y+z\|\le \|x,y\|+\|x,z\|$ for every $x,y,z\in X$.
\end{enumerate}
Then the function $\|.,.\|$ is called a \textit{2-norm} on $X$ and the pair $(X,\|.,.\|)$ a \textit{linear 2-normed space}. 
\subsubsection{Example} \emph{Let $X=\mathbb{R}^3$ and $x,y\in \mathbb{R}^3$ such that $x=(x_1,x_2,x_3)$ and $y=(y_1,y_2,y_3)$. Define
\begin{equation*}
\|x,y\|=\text{max}\, \{|x_1y_2-x_2y_1|+|x_1y_3-x_3y_1|,|x_1y_2-x_2y_1|+|x_2y_3-x_3y_2|\}
\end{equation*}
Then $(\mathbb{R}^3,\|.,.\|)$ is a 2-normed space.}
\subsection{Definition} (\cite{MU2}) Let $X$ be a non-empty set, and $\mathbb{R}^+$ denote the set of non-negative real numbers. Let $G\colon X\times X\times X \to \mathbb{R}^+$, be a function satisfying:
\begin{itemize}
\item [[G 1]] $G(x,y,z)=0$ if $x=y=z$,
\item [[G 2]] $G(x,x,y)>0$ for all $x$, $y \in X$ with $x\neq y$,
\item [[G 3]] $G(x,x,y)\leq G(x,y,z)$ for all $x$,$y$,$z\in X$, 
\item [[G 4]] $G(x,y,z)=G(y,z,x)=G(x,z,y)=\cdots$  for all $x$, $y$,$z\in X$, 
\item [[G 5]] $G(x,y,z)\leq G(x,a,a)+G(a,y,z)$ for all $x$,$y$,$z$,$a\in X$.
 
\end{itemize}
Then the function $G$ is called a generalized metric, or more specifically a \emph{$G$-metric} on $X$, and the pair $(X, G) $ a \emph{$G$-metric space}. 

\subsubsection{Example} \emph{Let $\mathbb{R} $ denote the set of all real numbers. Define a function\\ $\rho\colon \mathbb{R}\times \mathbb{R}\times \mathbb{R} \to \mathbb{R}^+$, by
\begin{equation*}
\rho(x,y,z)= \text{max} \{ \left|x-y\right|,\left|y-z\right|,\left|z-x\right|\}
\end{equation*}
for all $x$,$y$,$z\in \mathbb{R}$. Then $( \mathbb{R}, \rho)$ is a $G$-metric space.}

\section{main results}

\subsection{Definition}
Let $X$ be a real vector space. A real valued function $\|.,.,.\|\colon X^3\to \mathbb{R}$ is called a \textit{$G$-norm} on $X$ if the following conditions hold:
\begin{itemize}
\item [[N 1]] $\|x,y,z\|\ge 0$ and $\|x,y,z\|=0$ if and only if $x=y=z=0$ ,
\item [[N 2]] $\|x,y,z\|$ is invariant under permutation of $x$,$y$ and $z$,
\item [[N 3]] $\|\alpha x,\alpha y,\alpha z\|=|\alpha |\,\|x,y,z\|$ for all $\alpha \in \mathbb{R}$ and $x,y,z\in X$,
\item [[N 4]] $\|x+x',y+y',z+z'\|\le \|x,y,z\|+\|x',y',z'\|$ for all $x,y,z,x',y',z'\in X$,  
\item [[N 5]] $\|x,y,z\|\ge \|x+y,0,z\|$ for all $x,y,z\in X$. 
 
\end{itemize}
The pair $(X,\|.,.,.\|)$ is then called a \textit{$G$-normed space}.

\subsubsection{Example} \emph{In the linear space $X=C[0,1]$ of real valued continuous functions on $[0,1]$ define a function $\|.,.,.\|\colon X^3 \to \mathbb{R}$ by
\begin{equation*}
\|f,g,h\|=\max_{0\le t\le 1} \{|f(t)|+|g(t)|+|h(t)|\} \qquad (f,g,h\in C[0,1]).
\end{equation*}
 Then $( X, \|.,.,.\|)$ is a $G$-normed space.}

\subsubsection*{Remark}From [N 4], we see that $\|x+x',0,0\|\le \|x,0,0\|+\|x',0,0\|$.
This is similar to triangle inequality in usual normed spaces.

\begin{prop}
\label{prop1}
\emph{Let $(X,\|.,.,.\|)$ be a $G$-normed space. Then for all $x,y,z,u,v,w\in X$, we have}
\begin{equation}
\label{equ1}
|\,\|x,y,z\|-\|u,v,w\|\,|\le \|x-u,y-v,z-w\|
\end{equation} 
\end{prop}
\begin{proof}
The result follows directly from the definition of $G$-normed space.
\end{proof}

\begin{prop}
\label{prop2}
\emph{Let $(X,\|.,.,.\|)$ be a $G$-normed space. Then the function $G\colon X^3 \to \mathbb{R}^+$ defined by
\begin{equation}\
\label{equ4}
G(x,y,z)=\|x-y,y-z,z-x\|
\end{equation}
is a $G$-metric defined on $X$.}
\end{prop}
\begin{proof}
We see that [G 1] follows from [N 1]. Also $G(x,x,y)=\|0,x-y,y-x\|\ >0$ for $x\neq y$.
 From [N 2] and [N 5] we have
\begin{equation*}
\|x-y,y-z,z-x\|=\|y-z,z-x,x-y\|\ge \|y-x,0,x-y\|=\|0,x-y,y-x\|
\end{equation*} 
Which gives $G(x,y,z)\ge G(x,x,y)$. Now [G 4] follows from [N 2] and [N 3]. Finally [G 5] holds as we see that for all $x,y,z,a\in X$, we have 
\begin{align*}
G(x,y,z)&=\|x-a+a-y,0+y-z, a-x+z-a\|\\
&\le \|x-a,0,a-x\|+\|a-y,y-z,z-a\|\\
&=G(x,a,a)+G(a,y,z).
\end{align*}
Hence the function $G$ thus defined is a $G$-metric and $(X,G)$ is a $G$-metric space.
\end{proof}
Thus every $G$-normed space $(X,\|.,.,.\|)$ will be considered to be a $G$-metric space. We have a well defined topology for a $G$-metric space. For $x_0\in X$, $r>0$, the $G$-ball is defined as the set $B_G(x_0,r)=\{y\in X\colon G(x_0,y,y)<r\}$. The family of all $G$-balls $\{B_G(x,r)\colon x\in X,r>0\}$ is the base of a topology $\tau(G)$ on $X$, called the $G$-metric topology. This $G$-metric topology coincides with the metric topology corresponding to the metric $d_G$ given by $d_G(x,y)=G(x,y,y)+G(x,x,y)$ (\cite{MU2}). Thus every $G$-metric space and hence every $G$-normed space is topologically equivalent to a metric space. Now we can transport concepts such as open balls, open subsets, closed subsets, closure etc from metric spaces into the $G$-normed spaces.
\subsection{Definition}\emph{Let $(X,\|.,.,.\|)$ be a $G$-normed space. For given $x_0\in X$, $e\in X$ and $r>0$, we define \textit{open ball} $B_e(x_0,r)$ to be a subset of $X$ given by
\begin{equation}
\label{equ5}
B_e(x_0,r)=\{y\in X\colon \:\|x_0-y,y-e,e-x_0\|<r\}
\end{equation} 
and the \textit{closed ball} $B_e[x_0,r]$ in $X$ as
\begin{equation}
\label{equ6}
B_e[x_0,r]=\{y\in X\colon \:\|x_0-y,y-e,e-x_0\|\le r\}
\end{equation}}
Substituting $y=ru+x_0$ in ~\eqref{equ5}, we have
\begin{equation*}
B_e(x_0,r)=x_0+r\{u\in X\colon \: \|u, e'-u,-e'\|<1\}
\end{equation*}
where $e'=(e-x_0)r^{-1}$.\\
Hence for $x_0=0$, we have
\begin{align*}
B_e(0,r)&=r\,\{u\in X\colon \: \|u-0,\,e'-u,\, 0-e'\|<1\}\\
&=rB_{e'}(0,1)
\end{align*}
\subsubsection{Example}
\emph{Let $(\mathbb{R}^2,\|.,.,.\|)$ be a $G$-normed space such that
\begin{equation*}
\|x,y,z\|=\|x\|+\|y\|+\|z\|
\end{equation*}
for all $x,y,z\in \mathbb{R}^2$. Where $\|x\|=\|(x_1,x_2)\|=\sqrt{x_1^2+x_2^2}$.\\
Then the open ball $B_e(x_0,r)$ in $\mathbb{R}^2$ will be an open elliptic disc given by
\begin{equation*}
B_e(x_0,r)=\{y\in \mathbb{R}^2\colon \: \|x_0-y\|+\|y-e\|<k\}
\end{equation*}
Where $k=r-\|e-x_0\|$.}\\
\indent Suppose hereafter that $X$ is a $G$-normed space. Now we introduce some definitions and propositions for further theory.
\subsection{Definition}
\emph{A subset $G\subseteq X$ is \textit{open} in $X$ if for each $x\in G$, there exist $e\in X$ and $r>0$ such that $B_e(x,r)\subseteq G$}. 
\subsection{Definition}
\emph{A set $D$ in a $G$-normed space $X$ is said to be \textit{dense} in $X$ when it intersects every open set.}
\subsection{Definition}
\emph{A sequence $<x_n>$ in $X$ is said to be \textit{convergent} if there exists an element $x\in X$ such that for given $\epsilon >0$, there exists a positive integer $N$ such that
\begin{equation*}
l,m,n\ge N\Rightarrow \|x_l-x,\: x_m-x,\: x_n-x\|< \epsilon 
\end{equation*} 
Or equivalently $n\ge N\Rightarrow \|x_n-x,\: x_n-x,\: x_n-x\|< \epsilon$.}
\subsection{Definition}
\emph{A sequence $<x_n>$ in $X$ is said to be a \textit{Cauchy sequence} if given $\epsilon >0$, there exists a positive integer $N$ such that}
\begin{equation*}
l,m,n\ge N\Rightarrow \|x_l-x_m,\: x_m-x_n,\: x_n-x_l\|< \epsilon 
\end{equation*} 
\subsection{Definition}
\emph{A $G$-normed space is said to be complete if each Cauchy sequence in $X$ converges in $X$.}
\subsection{Definition}
\emph{A complete $G$-normed space is called a \textit{$G$-Banach space}.}
\subsection{Definition}
\emph{The \textit{closure} of a subset $E\subseteq X$, denoted by $\overline{E}$, is the set of all $x\in X$ such that there exists a sequence $<x_n>$ in $E$ converging to $x$. We say that $E$ is \textit{closed} if $E=\overline {E}$.}
\subsection{Definition}
\emph{A subset $C$ of $X$ is called convex (resp. absolutely convex) if $\alpha C+\beta C\subseteq C$ for every $\alpha ,\beta >0$ (resp. $\alpha ,\beta \in \mathbb{R}$) with $\alpha + \beta =1$ (resp. $|\alpha |+|\beta |\le 1$).}

\begin{prop}

\emph{Every convergent sequence in a $G$-normed space $X$ has a unique limit.}
\end{prop}

\begin{proof}
The proof is straightforward. 
\end{proof} 
\begin{prop}

\emph{Every convergent sequence in a $G$-normed space $X$ is a Cauchy sequence.}
\end{prop}

\begin{proof}
The result follows directly from the definitions.
\end{proof}

\begin{prop}
\label{prop3}
\emph{The ball $B_e(x_0,r)$ is open in $X$.}
\end{prop}
\begin{proof}
Let $z\in B_e(x_0,r)$. Then $\|x_0-e,e-z,z-x_0\|<r$.\\
Now using [N 5], we get
\begin{equation*}
r>\|x_0-e,e-z,z-x_0\|\ge \|x_0-z,0,z-x_0\|
\end{equation*}
Let $r_1=r-\|x_0-z,0,z-x_0\|$, then $r_1>0$. Now we shall show that $B_e(z,r_1)\subset B_e(x_0,r)$. Suppose $u\in B_e(z,r_1)$. Then we have $\|z-u,u-e,e-z\|<r_1$. Now
\begin{align*}
\|x_0-e,e-u,u-x_0\|&=\|x_0-z+z-e,e-u,z-x_0+u-z\|\\
&\le \|x_0-z,0,z-x_0\|+\|z-e,e-u,u-z\|\\
&<\|x_0-z,0,z-x_0\|+r_1=r\\
\Rightarrow \|x_0-e,e-u,u-x_0\|&<r
\end{align*}
Therefore $u\in B_e(x_0,r)$. Hence the result.
\end{proof}
\begin{prop}

\emph{For $a\in X$ and $r>0$, we have $\overline{B_e(a,r)}\subseteq B_e[a,r]$.}
\end{prop}

\begin{proof}
We prove the result by showing that $y \notin B_e[a,r] \Rightarrow y\notin \overline {B_e(a,r)} $.\\
 If $r_1=\|\,y-a,\: a-e, \: e-y\| $, then $r_1 >r$.\\
Let $x\in B_e(a,r)$. Then $\|x-a,a-e,e-x\| < r$. Now
\begin{align*}
\|\, y-a,\: a-e, \: e-y\|&=\|\, y-x+x-a,\:  a-e,\:  x-y+e-x\|\\
&\le \|y-x,0,x-y\|+ \|x-a, a-e, e-x\|\\
\Rightarrow \|\, y-x,\, 0,\, x-y\|& \ge \|\, y-a,\, a-e,\, e-y\|- \|\, x-a,\, a-e,\, e-x\|\\
& > r_1 -r >0
\end{align*}
Let $\epsilon =r_1-r$. Then $x \notin B_y(y,\epsilon )$. Hence there exists a neighborhood of $y$ which does not intersect $B_e(a,r)$, i.e. $y\notin \overline{B_e(a,r)}$. Hence the result.
\end{proof}
\begin{prop}

\emph{The balls $B_0(0,r)$ and $B_0[0,r]$ are absolutely convex for every $r > 0$.}
\end{prop}

\begin{proof}
Let $x,y \in B_0(0,r)$. Then $ \|\, x,\: -x, \: 0\, \|< r, \; \|\, y,\: -y, \: 0\, \|< r $. Let $\alpha ,\beta \in \mathbb{R} $ with $|\alpha |+|\beta|\le 1$. Then 
\begin{align*}
\|\, \alpha x+\beta y, \: -\alpha x-\beta y, \: 0\,\|& \le \|\, \alpha x, \: -\alpha x , \: 0\,\|+\|\, \beta y, \: -\beta y, \: 0\,\|\\
&=|\alpha |\,\|\,x, \, -x,\, 0\|+|\beta |\, \|\, y,\, -y,\, 0\|\\
&<(|\alpha |+|\beta |)\, r\\
&\le r
\end{align*}
This implies that $\alpha x+\beta y \in B_0(0,r)$. Hence $B_0(0,r)$ is absolutely convex. Similarly, we can show that the ball $B_0[0,r]$ is absolutely convex.
\end{proof}
\begin{prop}

\emph{The closure of a convex (resp. absolutely convex) subset of a $G$-normed space is convex (resp. absolutely convex).}
\end{prop}

\begin{proof}
Let $X$ be a $G$-normed space. Let $A$ be a convex (resp. absolutely convex) subset of $X$. Let $x,y\in \overline{A} $. Then there exist sequences $<x_n>$ and $<y_n>$ in $A$ such that $x_n \rightarrow x$ and $y_n \rightarrow y$.\\
Let $\alpha, \beta >0$ (resp. $\alpha, \beta \in \mathbb{R}$) such that $\alpha +\beta =1$ (resp. $ |\alpha|+|\beta|\le 1$). Since $A$ is convex (resp. absolutely convex), $\alpha x_n+ \beta y_n \in A$ for all $n\in \mathbb{N} $. Now
\begin{equation*}
\alpha x+\beta y =\lim_{n\to \infty} (\alpha x_n+\beta y_n ) \in \overline{A}
\end{equation*} 
Hence $A$ is convex (resp. absolutely convex).
\end{proof}
\begin{thm}
\emph{Let $X$ be a $G$-normed space. Then the following maps are continuous:
\begin{itemize}
\item[(a)] Addition : $X\times X\to X$, \: $(x,y)\longmapsto x+y$ ,
\item[(b)] Scalar multiplication : $ \mathbb{R}\times X \to X$,\: $(\lambda ,x)\longmapsto \lambda x$ ,
\item[(c)] The $G$-norm : $X\times X\times X \to \mathbb{R}$, \: $(x,y,z)\longmapsto \|x,y,z\|$.
\end{itemize}}
\end{thm}

\begin{proof}
We may regard $X$ and $\mathbb{R}$ as metric spaces. Let $<x_n>$, $<y_n>$ and $<z_n>$ be sequences in a $G$-normed space $X$ with $\lim x_n=x$, $\lim y_n=y$ and $\lim z_n=z$. Let $<a_n>$ be a sequence in $\mathbb{R}$ with $\lim a_n=a$.\\
(a) We have
\begin{align*}
\|x_l+y_l-(x+y), x_m+y_m-(x+y),x_n+y_n-(x+y)\|&\le \|x_l-x, x_m-x,x_n-x\|\\
&+\|y_l-y, y_m-y,y_n-y\|
\end{align*}
Hence $\lim_{l,m,n \to \infty} \|x_l+y_l-(x+y), x_m+y_m-(x+y),x_n+y_n-(x+y)\|=0$.\\
This proves the result.\\
(b) The result follows by similar arguments.\\
(c) Using the relation ~\eqref{equ1}, we have
\begin{equation*}
|\, \|x_n,y_n,z_n\|-\|x,y,z\|\, |\le \|\, x_n-x, y_n-y, z_n-z\, \|\\
\end{equation*}
Therefore $\lim_{n\to \infty} \|x_n,y_n,z_n\|=\|x,y,z\|$.
Hence the result.
\end{proof}
The following result holds for any topological vector space, hence we state it without proof.
\begin{prop}

\emph{The intersection of a finite number of dense open subsets of a $G$-normed space $X$ is dense in $X$.}
\end{prop}

\subsection{Definition}
\emph{A linear function $F$ from a $G$-normed space $(X,\|.,.,.\|_X)$ into a $G$-normed space $(Y,\|.,.,.\|_Y)$ is said to be bounded if there exists $K > 0$ such that 
\begin{equation}
\|\, F(x),\, F(y),\, F(z)\, \|_Y \le K\, \|\,x,\, y,\, z\, \|_X \; \: \forall \, x,y,z\in X
\end{equation}}
\subsection{Definition}
\emph{A linear function $F$ from a $G$-normed space $(X,\|.,.,.\|_X)$ into a $G$-normed space $(Y,\|.,.,.\|_Y)$ is continuous at $x_0\in X$ if there exists a sequence $<x_n>$ in $X$ such that
\begin{equation*}
x_n\to x_0 \Rightarrow F(x_n) \to F(x_0)
\end{equation*}
Equivalently, $F$ is continuous at a point $x_0$ if for given $z\in X$ and $\epsilon >0$ there exists $\delta >0$ and $u\in X$ such that $\|\, F(y)-F(x_0),\, F(x_0)-F(z),\, F(z)-F(y)\, \|< \epsilon $ for every $y\in X$ for which $\|\, y-x_0,\, x_0-u,\, u-y\|<\delta$.}\\
\indent $F$ is continuous if it is continuous at every point in $X$.
\begin{thm}

\emph{Every bounded linear function is continuous.}
\end{thm}

\begin{proof}
Let $F$ be a bounded linear function from a $G$-normed space $(X,\|.,.,.\|_X)$ into a $G$-normed space $(Y,\|.,.,.\|_Y)$. Then there exists $K > 0$ such that
\begin{equation*}
\|\, F(u),\, F(v),\, F(w)\, \|\le K\, \|\, u, \, v,\, w\, \| \; \; \forall \: u,v,w \in X
\end{equation*}
Hence for $x,y,z\in X$ and given $\epsilon >0$ there exists $\delta=\epsilon /K >0$ such that $\|\, F(y-x),\, F(x-z),\, F(z-x)\, \|< \epsilon$ whenever $ \|\, y-x, \, x-z,\, z-x\, \|<\delta $. Since $F$ is linear, hence $\|\, F(y)-F(x),\, F(x)-F(z),\, F(z)-F(x)\, \|< \epsilon$ whenever $ \|\, y-x, \, x-z,\, z-x\, \|<\delta $, i.e. $F$ is continuous.
\end{proof}
We now state and prove the famous Banach's fixed point theorem for $G$-normed spaces.
\begin{thm}
\label{Banach}
\emph{Let $X$ be a $G$-Banach space and $T \colon X \to X$ be a mapping satisfying the following condition for all $x,y,z \in X$ 
\begin{equation}
\label{equ7}
\|\, Tx-Ty,\, Ty-Tz, \, Tz-Tx\, \|\le k\, \|\,x-y,\,y-z,\,z-x\, \|
\end{equation}
Where $k \in [0,1)$. Then $T$ has a unique fixed point.}
\end{thm}

\begin{proof}
Let $T\colon X\to X$ be a mapping satisfying the condition ~\eqref{equ7}. Let $x_0\in X$ be an arbitrary point. Define a sequence $<x_n>$ by the relation $x_n=T^n x_0$, then by the given condition we have
\begin{equation*}
\|Tx_{n-1} -Tx_n,\, Tx_n-Tx_{n-1},\, 0\|< k\, \|x_{n-1}-x_n,\,x_n-x_{n-1},\,0\|
\end{equation*} 
or \qquad \quad    $\|x_n -x_{n+1},\, x_{n+1}-x_n,\, 0\|< k\, \|x_{n-1}-x_n,\,x_n-x_{n-1},\,0\|$.\\
Continuing the same argument, we have
\begin{equation}
\label{equ8}
\|x_n -x_{n+1},\, x_{n+1}-x_n,\, 0\|< k^n \, \|x_0-x_1,\,x_1-x_0,\,0\|
\end{equation}
For all natural numbers $n$ and $m (>n)$, by using [N 4] we have
\begin{align*}
\|x_n -x_m,\, x_m-x_n,\, 0\|\le \|x_n -x_{n+1},\, x_{n+1}-x_n,\, 0\|&+\|x_{n+1} -x_{n+2},\, x_{n+2}-x_{n+1},\, 0\|\\
& + \cdots +\|x_{m-1}-x_m,\, x_m-x_{m-1},\, 0\|
\end{align*}
Since $0\le k<1$, hence on using the relation~\eqref{equ8}, we get
\begin{align*}
\|x_n -x_m,\, x_m-x_n,\, 0\|&\le (k^n+k^{n+1}+ \cdots +k^{m-1})\,\|x_0-x_1,\,x_1-x_0,\,0\|\\
&\le \frac{k^n}{1-k}\,  \|x_0-x_1,\,x_1-x_0,\,0\| 
\end{align*}
This yields $\|x_n -x_m,\, x_m-x_n,\, 0\|\to 0$ as $m,n\to \infty$. Now
\begin{align*}
\|x_l -x_m,\, x_m-x_n,\, x_n-x_l\|&=\|x_l -x_m,\, x_m-x_l+x_l-x_n,\, x_n-x_l\|\\
&\le \,\|x_l -x_m,\, x_m-x_l,\, 0\|+\|0,\,x_l-x_n,\, x_n-x_l\|
\end{align*}
Therefore $\lim_{l,m,n\to \infty} \|x_l -x_m,\, x_m-x_n,\, x_n-x_l\|=0$ and hence $<x_n>$ is a Cauchy sequence. Since $X$ is complete, there exists $u\in X$ such that $x_n\to u$.\\
Suppose that $Tu\neq u$, then 
\begin{align*}
\|Tu-u,\, u-Tu,\, 0\|& \le \|Tu-Tx_n,\, Tx_n-Tu,\, Tu-Tu\|+\|Tx_n-u,\, u-Tx_n,\,0\|\\
&\le k\,\|u-x_n,\, x_n-u,\,0\|+\|x_{n+1}-u,\, u-x_{n+1},\, 0\|   
\end{align*}
 Taking the limits as $n\to \infty$ and using the fact that the $G$-norm is a continuous function of its variables, we observe that LHS is independent of $n$ and RHS tends to zero. Hence we must have $Tu=u$.\\
 For uniqueness of $u$, suppose that $v\neq u$ is such that $Tv=v$. Then we have
 \begin{equation*}
 \|v-u,\, 0,\,u-v\|=\|Tv-Tu,\, 0,\,Tu-Tv \|\le k\|v-u,\,0,\,u-v\|
 \end{equation*}
Which yields a contradiction as $0\le k <1$. Hence we have $u=v$.
\end{proof}
\subsection{Definition} \emph{Let $X$ be a $G$-normed space and $T$ be a self mapping on $X$. Then $T$ is called expansive mapping if there exists a constant $q>1$ such that for all $x,y,z \in X$, we have}
\begin{equation}
\label{equ9}
\|\, Tx-Ty,\, Ty-Tz, \, Tz-Tx\, \|\ge q\, \|\,x-y,\,y-z,\,z-x\, \|.
\end{equation}
\begin{thm}
\emph{Let $T$ be a linear surjective self mapping on a $G$-Banach space $X$ satisfying the condition~\eqref{equ9}. Then $T$ has a unique fixed point.}
\end{thm}
\begin{proof}
First we see that $T$ is invertible, for if $Tx=Ty$, taking $x,y,y$ for $x,y,z$, condition~\eqref{equ9} gives $x=y$, i.e. $T$ is injective and hence invertible.\\
Let $S$ be the inverse mapping of $T$. Then $S$ is linear and
\begin{equation*}
\|\,x-y,\,y-z,\,z-x\, \|=\|\,TS(x-y),\,TS(y-z),\,TS(z-x)\, \|\ge q\, \|\,Sx-Sy,\,Sy-Sz,\,Sz-Sx\, \|
\end{equation*}
or $\|\,Sx-Sy,\,Sy-Sz,\,Sz-Sx\, \|\le k\,\|\,x-y,\,y-z,\,z-x\, \|$ where $k=\frac{1}{q}$.\\
Hence by Theorem 2.17. the mapping $S$ has a unique fixed point $u\in X$ such that $Su=u$. Now $u=(TS)u=T(Su)=Tu$. Thus $u$ is also a fixed point of $T$.\\
If there exists some $v\neq u$ such that $Tv=v$, then $Tv=v=(TS)v=(ST)v=S(Tv)$, i.e. $Tv$ is another fixed point of $S$. By uniqueness of fixed point for $S$ we conclude that $u=Tv=v$, i.e. $u$ is a fixed point of $T$. 
\end{proof} 
\begin{thm}
\emph{Let $X$ be a $G$-Banach space and let $T$ and $S$ be self mappings on $X$ satisfying the following conditions:
\begin{enumerate}
\item $T(X)\subseteq S(X)$,
\item $S$ is continuous,
\item $\|Tx-Ty,\,Ty-Tz,\, Tz-Tx\|\le q\,\|Sx-Sy,\, Sy-Sz,\, Sz-Sx\|$ for every $x,y,z\in X$ and $0<q<1$.
\end{enumerate}
Then $T$ and $S$ have a unique common fixed in $X$ provided $T$ and $S$ commute.}
\end{thm}
\begin{proof}
Let $x_0$ be an arbitrary point in $X$. Since $T(X)\subseteq S(X)$ hence there exists a point $x_1$ such that $Tx_0=Sx_1$. In general we can choose $x_{n+1}$ such that $y_n=Tx_n=Sx_{n+1}$. From (3) we have
\begin{align*}
\|Tx_n-Tx_{n+1},\, Tx_{n+1}-Tx_n,\,Tx_n-Tx_n\|&\le q\, \|Sx_n-Sx_{n+1},\, Sx_{n+1}-Sx_n,\, Sx_n-Sx_n\|\\
&=q\, \|Tx_{n-1}-Tx_n,\, Tx_n-Tx_{n-1},\, 0\|
\end{align*}
Proceeding in above manner we have
\begin{align*}
\|Tx_n-Tx_{n+1},\, Tx_{n+1}-Tx_n,\,Tx_n-Tx_n\|&\le q^n \, \|Tx_0-Tx_1,\,Tx_1-Tx_0,\,0\| \\
\Rightarrow \|y_n-y_{n+1},\, y_{n+1}-y_n,\, 0\|&\le q^n\, \|y_0-y_1,\, y_1-y_0, \,0\|
\end{align*}
Hence for all natural numbers $n$ and $m(>n)$, it can be shown that
\begin{equation*}
\|y_n-y_m,\, y_m-y_n,\, 0\|\le \frac{q^n}{1-q}\,\|y_0-y_1,\, y_1-y_0, \,0\| 
\end{equation*}
This yields $\|y_n -y_m,\, y_m-y_n,\, 0\|\to 0$ as $m,n\to \infty$. Now
\begin{equation*}
\|y_l -y_m,\, y_m-y_n,\, y_n-y_l\|\le \,\|y_l -y_m,\, y_m-y_l,\, 0\|+\|0,\,y_l-y_n,\, y_n-y_l\|\rightarrow 0
\end{equation*}
as $l,m,n\to \infty$. Hence $<y_n>$ is a Cauchy sequence. Since $X$ is complete, there exists $u\in X$ such that $y_n\to u$. Since $y_n=Tx_n=Sx_{n+1}$, hence we have $\lim_{n\to \infty}y_n=\lim_{n\to \infty}Sx_n=\lim_{n\to \infty}Tx_n=u$. Now $S$ is continuous hence
\begin{equation*}
\lim_{n\to \infty}SSx_n=\lim_{n\to \infty}STx_n=Su
\end{equation*}
Also $T$ and $S$ commute, therefore
\begin{equation*}
\lim_{n\to \infty}TSx_n=\lim_{n\to \infty}STx_n=\lim_{n\to \infty}SSx_n=Su
\end{equation*}
Taking $x=Sx_n$, $y=x_n$ and $z=x_n$ in (3) we have
\begin{equation*}
\|TSx_n-Tx_n,\,0,\, Tx_n-TSx_n\|\le q\, \|SSx_n-Sx_n,\,0,\,Sx_n-SSx_n\| 
\end{equation*}
Making $n\to \infty$, we have $\|Su-u,\,0,\,u-Su\|\le q\, \|Su-u,\,0,\,u-Su\|$. Which gives $Su=u$. For otherwise $q\ge 1$ contradicting the fact that $0<q<1$.\\
Similarly on taking $x=x_n$, $y=u$ and $z=u$ in (3) and making $n\to \infty$, we have $Tu=u$. Therefore $Tu=Su=u$, i.e. $u$ is a common fixed point of $T$ and $S$.\\
For uniqueness of $u$, suppose that $v\neq u$ is such that $Tv=Sv=v$. Then we have 
\begin{equation*}
\|u-v,\,0,\,v-u\|=\|Tu-Tv,\,0,\,Tv-Tu\|\le q\, \|Su-Sv,\,0,\,Sv-Su\|<\|u-v,\,0,\,v-u\|
\end{equation*}
 
Thus we get a contradiction, hence we have $u=v$

\end{proof}


\begin{thebibliography}{999}
\bibitem{DHA} B.C. Dhage,  \textit{ A study of some fixed point theorem}, Ph.D. Thesis, Marathwada Univ. Aurangabad, (1984).\\

\bibitem{FR1} R. Freese, Y.J. Cho, and S. S. Kim, \textit{Strictly 2-convex linear 2-normed spaces}, J. Korean Math. Soc. \textbf{29} (1992), 391-400.\\

\bibitem{FR2} R. Freese and Y.J. Cho, \textit{Geometry of  Linear 2-normed Spaces}, Nova Science Publ., New York, (2001).\\

\bibitem{GA1} S. G\"{a}hler, \textit{ 2-metrische r\"{a}ume und ihre topologische struktur}, Math. Nachr. \textbf{26} (1963), 115-148.\\

\bibitem{GA2} S. G\"{a}hler, \textit{Lineare 2-normierte r\"{a}ume}, Math. Nachr. \textbf{28} (1964), 1-43.\\

\bibitem{GA3} S. G\"{a}hler, \textit{Untersuchungen \"{u}ber verallgemeinerte m-metrische R\"{a}ume. I}, Math. Nachr. \textbf{40} (1969), 165-189.\\

\bibitem{GA4} S. G\"{a}hler, \textit{Untersuchungen \"{u}ber verallgemeinerte m-metrische R\"{a}ume. II}, Math. Nachr. \textbf{40} (1969), 229-264.\\

\bibitem{GA5} S. G\"{a}hler, \textit{Untersuchungen \"{u}ber verallgemeinerte m-metrische R\"{a}ume. III}, Math. Nachr. \textbf{41} (1969), 23-36.\\


\bibitem{GUN} H. Gunawan, and Mashadi \textit{On $n$-normed spaces}, Int. J. Math. Math. Sci. \textbf{27} (2001) 631-639.\\

\bibitem{HA1} K.S. Ha, Y.J. Cho and A. White, \textit{strictly convex and 2-convex 2-normed spaces}, Math. Japonica  \textbf{33} (1988), 375-384.\\


\bibitem{KHA} K.A. Khan, \textit{Generalized $n$-metric spaces and fixed point theorems}, to appear in Journal of Nonlinear and Convex Analysis. \\
 
\bibitem{MIS} A. Misiak, \textit{$n$-inner product spaces}, Math.Nachr. \textbf{140} (1989), 299-319. \\

\bibitem{MU1}Z. Mustafa, and B. Sims, \textit{Some remarks concerning D-metric spaces}, Proceedings of the International Conferences on Fixed Point Theory and Applications, Valencia, Spain, July (2003), 189-198.\\

\bibitem{MU2} Z. Mustafa, and B. Sims, \textit{ A new approach to generalized metric spaces}, Journal of Nonlinear and Convex Analysis  \textbf{7} (2006), 289-297.\\

\bibitem{RIY} P. Riyas and K.T. Ravindran, \textit{2-NSR lemma and quotient space in 2-normed space}, Mat. Vesnik
\textbf{63} (2011), 1-6.\\

\bibitem{SHA} A.K. Sharma, \textit{A note on fixed points in 2-metric spaces}, Indian J. Pure Appl. Math. \textbf{11} (1980), 1580-1583.

\end{thebibliography}
\end{document}